\documentclass[fleqn,11pt]{article}
\usepackage{amsmath,amssymb,amsfonts,amsthm,graphicx}
\usepackage{color}

\newtheorem{Theorem}{Theorem}

\newtheorem{Proposition}[Theorem]{Proposition}
\newtheorem{Definition}[Theorem]{Definition}
\newtheorem{Corollary}[Theorem]{Corollary}
\newtheorem{Lemma}[Theorem]{Lemma}
\newtheorem{Remark}[Theorem]{Remark}

\begin{document}

\title{\textbf{On \( p \)-hyponormal operators on \\quaternionic Hilbert spaces}}
\author{ M. Fashandi\footnote{{\small{E-mail:}  fashandi@um.ac.ir (M. Fashandi)}} \\
\footnotesize{  Faculty of Mathematical Sciences, Ferdowsi University of Mashhad,}\\  
  \footnotesize{P.O. Box 1159,  Mashhad, 91775\ Iran }
  }

\date{}
\maketitle

\begin{abstract}

This paper extends the notion of a \( p \)-hyponormal operator for a bounded right linear quaternionic operator defined on a right quaternionic Hilbert space. Several fundamental properties of complex \( p \)-hyponormal operators are investigated for the quaternionic ones. To develop the results, we prove the well-known Furuta inequality for quaternionic positive operators. This inequality opens the way to discuss the \( p \)-hyponormality of a quaternionic operator and its Aluthge transform. Finally, a new class of quaternionic operators is established between quaternionic \( p \)-hyponormal and quaternionic paranormal operators.

\end{abstract}

{\it 2020 Mathematics Subject Classification MSC: 47A63, 47S05 46S05}  \\
\maketitle

\section{Introduction}
``All analysts spend half their time hunting through the literature for inequalities which they want to use and cannot prove", said  G. H. Hardy on November 8, 1928, during his presidential address to the London Mathematical Society~\cite{Fink}. Both mathematical analysts and humans have been using inequalities for thousands of years, dating back to the Mesopotamian and Egyptian civilizations, to measure land and construct various projects. The earliest inequalities known to the ancient Greeks included the triangle inequality and the arithmetic-geometric mean inequality for two numbers. For a detailed history of inequalities in mathematics, refer to \cite{Fink}. Inequalities are crucial in establishing bounds and defining key mathematical concepts such as limit definitions, convex functions, and positive operators. Named inequalities, like the Cauchy-Schwarz and Hölder's inequalities, are significant results attributed to individuals who proved them and are recognized for their importance. Over time, new inequalities are established, some old ones refined, and others find new applications in science and engineering. In various mathematical disciplines, new tools extend inequalities to different spaces. The study of operator inequalities in Hilbert spaces is one of the most fascinating areas of mathematical analysis. It has captivated the interest of numerous mathematicians over the past century (see, for instance, \cite{Drag2}, \cite{Drag}, \cite{Huy}, \cite{Kit}, and \cite{Mond}, along with the references provided there).

Many papers have been published extending operator inequalities to spaces similar to or derived from Hilbert spaces, such as Hilbert \( C^* \)-modules and direct sum of Hilbert spaces. Additionally, these extensions may involve different types of operators (see \cite{Fuji} and \cite{Drag4}). Quaternionic Hilbert spaces are currently a popular topic in mathematical physics, with numerous applications in quantum systems (for further details, see \cite{Ghi1} and the references therein). To study operator inequality in quaternionic Hilbert spaces, one must be familiar with spherical spectrum and continuous functional calculus in a quaternionic setting. Colombo et al. in \cite{Col-book} defined the spherical spectrum of operators on two-sided quaternionic Banach spaces. Later, in \cite{Colombo} and \cite{Ghi1}, this concept was re-adapted in quaternionic Hilbert spaces. Also, Ghiloni et al., in \cite{Ghi1}, present the definition of continuous functional calculus in quaternionic Hilbert spaces. Using these tools, some operator inequalities in quaternionic Hilbert spaces are derived in \cite{Adv.} and \cite{Kont}. This paper aims to define quaternionic \( p \)-hyponormal operators and analyze their Alutghe transform. We extend some fundamental properties of complex \( p \)-hyponormal operators to their quaternionic counterparts. To achieve this, it is essential to apply various quaternionic operator inequalities, including the Quaternionic Hölder-McCarthy inequality and the Quaternionic  L\"{o}wner-Heinz inequality, both of which were established by the author in \cite{Adv.}. Additionally, the Quaternionic Furuta inequality is presented and proved in this paper.


This paper is structured as follows: Section 2 will review the essential definitions and theorems for the proofs. Section 3 introduces the concept of quaternionic \( p \)-hyponormal operators and explores their fundamental properties. This topic is closely related to the polar decomposition of an operator. Therefore, it is essential to study some transformations based on the polar decomposition of a quaternionic operator, such as its quaternionic Aluthge transform. One of the most intriguing findings from this section is the generalization of the Furuta inequality for quaternionic positive operators. Finally, Section 4 extends the generalized Cauchy-Schwarz inequality for quaternionic operators and demonstrates that any quaternionic operator satisfying this inequality is paranormal.

\section{Preliminaries and Auxiliary Results}


In this paper, we will denote the skew field of quaternions by $\mathbb{H}$. The elements of $\mathbb{H}$ can be expressed in the form \( q = x_0 + x_1i + x_2j + x_3k \), where \( x_0, x_1, x_2, \) and \( x_3 \) are real numbers. In this context, \( i, j, \) and \( k \) are referred to as imaginary units, which follow specific multiplication rules as follows:
\begin{equation*}
i^2=j^2=k^2=-1,\;\; ij=-ji=k,\;\; jk=-kj=i,\;\; {\rm\mbox and}\;\; ki=-ik=j.
\end{equation*}
Also, $\mathsf{H}$ denotes a right quaternionic Hilbert space, which is a linear vector space over the skew field of  quaternions under right scalar multiplication and an inner product 
$\langle.,.\rangle:\mathsf{H}\times \mathsf{H}\longrightarrow \mathbb{H}$
with the following properties:
\begin{enumerate}
\item[(i)] $\overline{\langle u, v\rangle}=\langle v, u\rangle$,
\item[(ii)] $\langle u, u\rangle>0$ unless $u=0$,
\item[(iii)] $\langle u,vp+wq\rangle=\langle u, v\rangle p+\langle u, w\rangle q$,
\end{enumerate}
for every $u, v, w\in \mathsf{H}$ and $p, q\in \mathbb{H}$. The quaternionic norm  $\Vert u \Vert= \sqrt{\langle u, u\rangle}$
is assumed to produce a complete metric space (see  Proposition 2.2 in  \cite{Ghi1}). 
Like the complex Hilbert spaces, every quaternionic Hilbert space admits a Hilbert basis. This fact can be found in Propositions 2.5 and 2.6 in \cite{Ghi1}.
\begin{Proposition}(Propositions 2.5 and 2.6 of \cite{Ghi1})\label{Hbasis}
Let $\mathsf{H}$ be a right  quaternionic Hilbert space and let $\mathsf{N}$ be a subset of $\mathsf{H}$ such that, for $z, z'\in \mathsf{N}$,
$\langle z,z'\rangle=0$ if $z\neq z'$ and $\langle z,z\rangle=1$. Then conditions (a)-(e) listed below are pairwise equivalent.
\begin{itemize}
\item[(a)] For every $u, v\in \mathsf{H}$ the series $\sum_{z\in \mathsf{N}}\langle u,z\rangle \langle z,v\rangle$ converges absolutely to $\langle u,v\rangle$.
\item[(b)] every $u\in\mathsf{H}$ can be uniquely decomposed as $u=\sum_{z\in \mathsf{N}} z \langle z,u\rangle$, where the series  converges absolutely in $\mathsf{H}$.
\item[(c)] $\Vert u\Vert^2=\sum_{z\in \mathsf{N}} \vert \langle z,u\rangle\vert^2$ for every $u\in \mathsf{H}$.
\item[(d)] $\mathsf{N}^\perp:=\{v\in\mathsf{H}: \langle v,z\rangle=0,\,\forall z\in\mathsf{N}\}=\{0\}.$
\item[(e)] $\langle \mathsf{N}\rangle$ is dense in $\mathsf{H}.$
\end{itemize}
\end{Proposition}
\noindent The subset $\mathsf{N}$ in Proposition \ref{Hbasis} is called a Hilbert basis.\\
\noindent A right quaternionic Hilbert space $\mathsf{H}$ can be equipped with a left multiplication with quaternions. According to  Section 3.1 in \cite{Ghi1}, the left scalar multiplication on $\mathsf{H}$ induced by a Hilbert basis $\mathsf{N}$ of $\mathsf{H}$ is defined by
\begin{equation}
\label{left multiplication}
qu:= \sum_{z\in \mathsf{N}}zq\langle z,u\rangle \,\,\,\mbox{if} \,\,\,u\in \mathsf{H} \,\,\, \mbox{and} \,\,\,\, q\in \mathbb{H}.
\end{equation}
We refer the interested reader to Proposition 3.1 in \cite{Ghi1} for the properties of the left multiplication defined by (\ref{left multiplication}).
\noindent The notation $\mathfrak{B}(\mathsf{H})$ shows the set of all bounded right linear quaternionic operators on $\mathsf{H}$ that means 
 $$T(up+v)=(Tu)p+Tv,$$
for all $u, v \in \mathsf{H}$ and $p\in \mathbb{H}$. By  Proposition 2.11 in \cite{Ghi1}, $\mathfrak{B}(\mathsf{H})$  is a complete normed space with the norm defined by 
\begin{equation*}\label{norm1}
\Vert T\Vert= \sup \left\{\dfrac{\Vert Tu\Vert}{\Vert u\Vert}, 0\neq u\in\mathsf{H}\right\}.
\end{equation*}
Adjoint of $T\in \mathfrak{B}(\mathsf{H})$, selfadjoint, positive, normal, and unitary right linear quaternionic operators in $\mathfrak{B}(\mathsf{H})$ are defined similarly to the complex case (see Definition 2.12 in \cite{Ghi1}).
In Section 4 in \cite{Ghi1}, one can see the extension of the definitions of spectrum and resolvent in $\mathsf{H}$, as follows.
\begin{Definition}\label{spectrum}\cite{Ghi1}
Let $\mathsf{H}$ be a right quaternionic Hilbert space and $T\in \mathfrak{B}(\mathsf{H})$ be a right linear quaternionic operator. For $q\in \mathbb{H}$, the associated operator $\Delta_{q}(T)$ is defined by:
$$
\Delta_q(T):=T^2- T(q+\bar{q})+I \vert q\vert^2.
$$
The spherical resolvent set of $T\in \mathfrak{B}(\mathsf{H})$ is the set
$$
\rho_{S}(T):=\{q\in\mathbb{H}:\,\,\, \Delta_q(T)^{-1}\in \mathfrak{B}(\mathsf{H})\}.
$$
\end{Definition}
\noindent The \textit{spherical spectrum} $\sigma_{S}(T)$ of $T$, or briefly the S-spectrum, is defined as the complement of $\rho_{S}(T)$ in $\mathbb{H}$. A partition for $\sigma_{S}(T)$ was introduced in \cite{Ghi1}, as follows:
\begin{itemize}
\item[(i)] The \textit{spherical point spectrum} of $T$:
$$
\sigma_{pS}(T)=\{q\in\mathbb{H}; \ker(\Delta_{q}(T))\neq\{0\}\}.
$$
\item[(ii)] The\textit{ spherical residual spectrum} of $T$:
$$\sigma_{rS}(T)=\{q\in\mathbb{H}; \ker(\Delta_{q}(T))=\{0\}, \overline{Ran (\Delta_{q}(T))}\neq\mathsf{H}\}.$$
\item[(iii)] The spherical continuous spectrum of $T$:
$$\sigma_{cS}(T)=\{q\in\mathbb{H}; \ker(\Delta_{q}(T))=\{0\}, \overline{Ran (\Delta_{q}(T))}=\mathsf{H}, 
\Delta_{q}(T)^{-1}\notin \mathfrak{B}(\mathsf{H})\}.$$
\end{itemize}
The \textit{spherical spectral radius} of $T$, denoted by $r_{S}(T)$,  is defined by:
\begin{equation*}\label{radius}
r_{S}(T)=\sup\{\vert q\vert \in \mathbb{R}^{+}; q\in \sigma_{S}(T)\}.
\end{equation*}
An \textit{eigenvector} of $T$ with \textit{right} \textit{eigenvalue} $q$ is an element $u\in \mathsf{H}-\{0\}$, for which $Tu=uq$.
According to Proposition 4.5 in \cite{Ghi1}, $q$ is a right eigenvalue of $T$ if and only if it belongs to the spherical point spectrum $\sigma_{pS}(T)$.

\noindent In the sequel, we need the notion of quaternionic continuous functional calculus (see Theorem 5.5 in \cite{Ghi1}).
\begin{Theorem}\label{remark}\cite{Ghi1}
Let $\mathsf{H}$ be a right quaternionic Hilbert space and  $T\in\mathfrak{B}(\mathsf{H})$ be selfadjoint. Then, there exists a unique continuous homomorphism 
\begin{eqnarray*}
\Phi_T: \mathcal{C}(\sigma_S(T), \mathbb{R})\ni f\mapsto f(T)\in \mathfrak{B}(\mathsf{H}),
\end{eqnarray*}
of real Banach unital algebras such that:
\begin{itemize}
\item[(i)] the operator $f(T)$ is selfadjoint for every $f\in \mathcal{C}(\sigma_S(T), \mathbb{R})$.
\item[(ii)] $\Phi_T$ is positive; that is, $f(T)\geq 0$ if $f \in  \mathcal{C}(\sigma_S(T), \mathbb{R})$ and $f(t)\geq 0$ for every $t\in \sigma_S(T)$.
\end{itemize}
By $\mathcal{C}(\sigma_S(T), \mathbb{R})$ we mean the commutative real Banach unital algebra of continuous real-valued functions defined on $\sigma_S(T)$. 
\end{Theorem}
\begin{Remark}
\begin{itemize}
\item[(i)] By Proposition 2.17 in \cite{Ghi1}, a positive operator $T$,  is selfadjoint and  by Lemma 6.2 in \cite{Ghi2}, $\sigma_{S} (T) \subset [0,+\infty)$. Conversely, taking $f(t)=\sqrt{t}$ on $\sigma_{S} (T)\subset [0,+\infty)$, for a selfadjoint operator $T$, Theorem \ref{remark} (iii), implies that $T$ is a positive quaternionic operator. 
\item[(ii)]
Theorem \ref{remark}. (ii) implies that  for a selfadjoint quaternionic operator $T\in\mathfrak{B}(\mathsf{H})$, if $f, g \in  \mathcal{C}(\sigma_S(T), \mathbb{R})$ and $f(t)\geq g(t)$ for every $t\in \sigma_S(T)$, then $f(T)\geq g(T)$.
\end{itemize}
\end{Remark}
\noindent  For  two quaternionic operators $S$ and $T \in \mathfrak{B}(\mathsf{H})$, we write $S \leq T$ if $T-S$ is a positive quaternionic  operator, i.e. $\langle Sx, x\rangle \leq \langle Tx, x\rangle$ for every $x\in \mathsf{H}$. In particular,  for two real numbers $m<M$, by $mI\leq T\leq MI$,  we mean
$
m\langle x, x\rangle\leq \langle Tx, x\rangle\leq M\langle x, x\rangle
$
for all $x\in \mathsf{H}$.

\noindent The following theorem from \cite{Adv.} establishes the optimum bounds for the S-spectrum of a selfadjoint quaternionic operator. Also, it proves that the S-spectrums of $ST$ and $TS$, along with zero and their  S-spectral radiuses are the same. This theorem is vital for the next result. 
\begin{Theorem}\label{compact interval} Let $\mathsf{H}$ be a right quaternionic Hilbert space. For $S, T\in\mathfrak{B}(\mathsf{H})$, we have
\begin{itemize} 
\item[(i)]
If $T$ is a selfadjoint quaternionic operator and
\begin{equation*}
\label{infsup}
m_T=\inf \{\langle Tx,x\rangle\,:  \Vert x\Vert=1,\,x\in \mathsf{H}\}, \, \mbox{and}\,\, M_T=\sup \{\langle Tx,x\rangle\, : \Vert x\Vert=1,\,x\in \mathsf{H}\},
\end{equation*}
then $\sigma_S(T)\subset [m_T, M_T]$, $m_T=\min \sigma_S(T)$ and  $M_T=\max \sigma_S(T)$(see Theorem 3 in \cite{Adv.}).
\item[(ii)]
$\sigma_S(ST)\cup \{0 \}=\sigma_S(TS)\cup \{0 \}$ and $r_S(ST)=r_S(TS)$ (see Theorem 5  in \cite{Adv.}).
\end{itemize}
\end{Theorem}
\noindent From now on, $\mathsf{H}$ will stand for a right quaternionic Hilbert space, and by an operator $T\in \mathfrak{B}(\mathsf{H})$, we mean a bounded right linear quaternionic operator.
\noindent The following lemma is an extension of Lemma 1.6 in \cite{Mond} for the quaternionic setting.  
\begin{Lemma} \label{U^*SU}
Let $S\in\mathfrak{B}(\mathsf{H})$ be selfajoint, and $U\in\mathfrak{B}(\mathsf{H})$ be a partial isometry , i.e. $U^*U=I$ on a closed subset of $\mathsf{H}$. Then for every $f\in C(\sigma_S(S))$
$$f(USU^*)=Uf(S)U^*. $$
\end{Lemma}
\begin{proof}
Put $B=USU^*$. Obviously $B$ is selfadjoint and by Theorem \ref{compact interval} (ii), we have  
$$\sigma_S (B)\cup \{0 \}=\sigma_S(USU^*)\cup \{0 \}=\sigma_S(SU^*U)\cup \{0 \}=\sigma_S(S)\cup \{0 \}.$$
It is clear that $B^m=US^mU^*$ for every $m\in \mathbb{N}$, hence, $P(B)=UP(S)U^*$ for every polynomial $P$. Theorem 4.3 (c) in \cite{Ghi1}, implies that $\sigma_S(P(B))=P(\sigma_S(B))=P(\sigma_S(S))$. Now, for the given $f\in C(\sigma_S(S))$  there exists a sequence of polynomials $\{P_j\}_{j\in \mathbb{N}}$ such that $\Vert f-P_j\Vert\rightarrow 0$ as $j\rightarrow +\infty$. Therefore,
\begin{align*}
\Vert f(USU^*)-Uf(S)U^*\Vert &\leq \Vert f(USU^*)- P_j(USU^*)\Vert + \Vert P_j(USU^*)- U P_j(S)U^* \Vert \\
 &  + \Vert U P_j(S)U^* -Uf(S)U^*\Vert \rightarrow 0,\,\,\,\mbox{as}\, j\rightarrow +\infty.
\end{align*}
 So $f(USU^*)=Uf(S)U^*$.
\end{proof}
\noindent According to Theorem 2.20 in \cite{Ghi1},  for $T \in  \mathfrak{B}(\mathsf{H})$ there exists a unique polar decomposition $T=U\vert T\vert $, where $\vert T\vert= (T^*T)^{\frac{1}{2}}$ and $\Vert Uu\Vert= \Vert u\Vert$ for every $u\in \ker (\vert T\vert)^{\perp}$, which means that $U$ is a partial isomerty on $\ker (\vert T\vert)^{\perp}$. Also, $\ker (\vert T\vert) =\ker (T)$. \\
The following corollary, which will be used several times during the rest of the paper, establishes the polar decomposition of $T^*$ based on the polar decomposition of $T$.
\begin{Corollary}\label{T^*}
Let $T \in  \mathfrak{B}(\mathsf{H})$ and $T=U\vert T\vert$ be its polar decomposition. Then 
\begin{itemize}
\item[(i)] $\vert T^* \vert^t= U\vert T\vert ^tU^*$ for all positive numbers $t$.
\item[(ii)] $T^*= U^* \vert T^*\vert$ is the polar decomposition of $T^*$.
\end{itemize}
\end{Corollary}
\begin{proof}
\begin{itemize}
\item[(i)] Since $\vert T\vert$ is a positive operator, it is  selfadjoint, thus $T^*= \vert T\vert U^*$. Clearly, 
\begin{equation}\label{x^t}
\vert T^* \vert ^2= TT^*=U\vert T\vert \vert T\vert U^*=U\vert T\vert ^2 U^*.
\end{equation}
Let $t > 0$ and  $f(x)=x^{\frac{t}{2}}$  for  $x\in \sigma_S(\vert T\vert ^2) \subset (0, +\infty)$. Now, using Lemma \ref{U^*SU}, we obtain
$(U|T|^2 U^*)^{\frac{t}{2}}= U(\vert T\vert ^2)^{\frac{t}{2}}U^*$. We obtain the result by substituting (\ref{x^t}) in the latter equality.
\item[(ii)] By (i) we have $\vert T^* \vert= U\vert T\vert U^*$ and since $\Vert Uu\Vert= \Vert u\Vert$ for every $u\in \ker (\vert T\vert)^{\perp}$, we infer that $U^*U=I$ on $(\ker )^{\perp}$. Thus $T^*=\vert T\vert U^*=U^*U \vert T\vert U^*=U^* \vert T^*\vert$. Also,  $\Vert U^*u\Vert= \Vert u\Vert$ for every $u\in \ker (\vert T^*\vert)^{\perp}$.
\end{itemize}
\end{proof}
\noindent There are two quaternionic operator inequalities, quaternionic H\"{o}lder-McCarthy's inequality and quaternionic  L\"{o}wner-Heinz inequality,  that play crucial roles in the proof of the main results of this paper. We repeat them here from \cite{Adv.} (Theorems 13 and 14 in \cite{Adv.}).
\begin{Theorem}\label{H-M-L}
Let $S, T\in\mathfrak{B}(\mathsf{H})$ and $S\geq T\geq 0$, then for all $x\in \mathsf{H}$,
\begin{itemize}
\item[(i)] Quaternionic H\"{o}lder-McCarthy's inequality: 
\begin{itemize}
\item[(a)]
$\langle T^r x, x\rangle \geq \langle Tx, x\rangle ^r \Vert x\Vert ^{2(1-r)}$ for all $r>1$,
\item[(b)] $\langle T^r x, x\rangle \leq \langle Tx, x\rangle ^r \Vert x\Vert ^{2(1-r)}$ for all $0<r<1$.
\end{itemize}
\item[(ii)] Quaternionic  L\"{o}wner-Heinz inequality: $S^r\geq T^r$ for all $r\in [0,1]$.
\end{itemize}
\end{Theorem}


\section{\( p \)-hyponormal Quatenionic Operators}


To the authors' knowledge, this is the first time that quaternionic  \( p \)-hyponormal operators are studied. \( p \)-hyponormal operators on complex Hilbert spaces have been popular for at least four decades (see \cite{Alu} and the references therein), and still, many papers are published in this direction (see, for instance, \cite{B24}, \cite{Srt}). This section extends the notion of \( p \)-hyponormality to the quaternionic setting and investigates the properties of quaternionic \( p \)-hyponormal operators.  
An operator $T\in\mathfrak{B}(\mathsf{H})$ is said to be quaternionic \( p \)-hyponormal if $(T^*T)^p\geq (TT^*)^p, 0<p\leq 1$. In particular, if $p=1$ and $p=\frac{1}{2}$, then $T$ is called the quaternionic hyponormal and quaternionic semi-hyponormal operators, respectively. An operator $T\in \mathfrak{B}(\mathsf{H})$ is called quaternionic paranormal if $\Vert Tx\Vert^2 \leq \Vert T^2x\Vert \Vert x\Vert$ for every $x\in \mathsf{H}$. From now on, for brevity, we omit the word quaternionic from the phrases mentioned above.
\noindent Clearly, every normal operator is 1-hyponormal and if $T\in \mathfrak{B}(\mathsf{H})$ is \( p \)-hyponormal then it is paranormal. 
\begin{Proposition}\label{p-q}
Every  \( p \)-hyponormal operator is q-hyponormal if $0<q\leq p \leq 1$.
\end{Proposition}
\begin{proof}
If $0<q\leq p \leq 1$ and $T\in \mathfrak{B}(\mathsf{H})$ is a \( p \)-hyponormal operator, then $ (T^*T)^p\geq (TT^*)^p$. Note that $T^*T$ and $TT^*$ are positive operators (see the proof of Proposition 2.19 in \cite{Ghi1}). Then, by Theorem \ref{H-M-L}. (ii), for $0<\frac{q}{p}\leq 1$, we have
$$ ((T^*T)^p)^{\frac{q}{p}}\geq ((TT^*)^p)^{\frac{q}{p}},$$
that proves $T$ is q-hyponormal.
\end{proof}
\noindent  \( p \)-hyponormality of bounded operators on complex Hilbert spaces is closely related to their polar decomposition. Aluthge in \cite{Alu} for the first time found a relationship between \( p \)-hyponormality of a bounded operator $T=U\vert T\vert $ and the operator $\tilde{T}= \vert T\vert^\frac{1}{2}U\vert T\vert^\frac{1}{2}$,  called the Aluthge transform after that.  In \cite{landaAlu}, a general form of the Aluthge transform  was introduced  by $\Delta_{\lambda}(T)=\vert T\vert^{\lambda}U\vert T\vert^{1- \lambda}$ for $\lambda \in [0,1]$, called the $\lambda$-Aluthge transform of $T$. These transforms not only follow the properties of the original operator but also present more interesting features. Another transform related to the polar decomposition of $T$ is called the Duggal transform of $T$,  defined in \cite{Duggal} as  $\tilde {T}^D=\vert T\vert U$.  
Furuta in \cite{FSRT} studied the properties of  $S_r(T)= U\vert T\vert ^rU$ that was a generalization of the transform $S(T)= U\vert T\vert ^\frac{1}{2}U$, introduced in \cite{Patel}. \\
\noindent It would be an interesting project to study all these transforms for the quaternionic bounded operators. Let's start with the quaternionic Aluthge transform of $T\in \mathfrak{B}(\mathsf{H})$. In the next theorem, we prove the quaternionic version of Theorem 1 in \cite{Alu}.
\begin{Theorem}\label{Q-ALU}
If $T\in \mathfrak{B}(\mathsf{H})$ is a \( p \)-hyponormal operator for $p\in [\frac{1}{2}, 1)$, and $T=U\vert T\vert $ is its polar decomposition, then the quaternionic Aluthge transform of $T$ is hyponormal.
\end{Theorem}
\begin{proof}
By Proposition \ref{p-q}, $T$ is semi-hyponormal, i.e. $(T^*T)^{\frac{1}{2}}\geq (TT^*)^{\frac{1}{2}}$ and  $T^*T\geq TT^*$, by Theorem \ref{H-M-L}. (ii), that means $\vert T\vert \geq \vert T^*\vert$. From Corollary \ref{T^*}. (i), we see that $\vert T\vert \geq U \vert T\vert U^*$. Again,  using  Corollary \ref{T^*}. (ii) and the inequality $\vert T\vert \geq \vert T^*\vert$  we obtain
\begin{equation*}
\vert T\vert ^2= T^*T= U^*\vert T^*\vert T\leq U^* \vert T\vert T=U^* \vert T\vert  U  \vert T\vert.
\end{equation*}
Therefore $U^* \vert T\vert  U  \vert T\vert -\vert T\vert ^2 =(U^* \vert T\vert  U -\vert T\vert )  \vert T\vert$ is a positive operator and since $  \vert T\vert \geq 0$, the operator $U^* \vert T\vert  U -\vert T\vert$ is positive. All together give the following inequality
\begin{equation}\label{U*TU>|T|}
U^* \vert T\vert  U \geq \vert T\vert \geq U \vert T\vert U^*.
\end{equation}
Finally, $\tilde{T^*}\tilde{T}-\tilde{T}\tilde{T^*}=\vert T\vert ^{\frac{1}{2}}(U^* \vert T\vert  U- U \vert T\vert U^*)\vert T\vert ^{\frac{1}{2}} \geq 0$, by (\ref{U*TU>|T|}), that completes the proof.
\end{proof}
\noindent Aluthge in Theorem 2 of \cite{Alu} proved that for $p\in (0, \frac{1}{2})$ the Aluthge transform of the \( p \)-hyponormal operator $T$  on a complex Hilbert space is $(p+\frac{1}{2})$-hyponormal. This result is valid in the quaternionic setting if only we prove the Furuta inequality for quaternionic positive operators on a right quaternionic Hilbert space $\mathsf{H}$. Furuta inequality in complex cases was proved in \cite{Furutainequality}.
\begin{Theorem} (Quaternionic Furuta inequality)\label{QFI}
Let $A, B\in \mathfrak{B}(\mathsf{H})$ with $A\geq B\geq 0$ and  $r\geq 0$. If $p\geq 0 $ and $ q\geq 1$ are such that $(1+2r)q \geq p+2r$ then
\begin{equation}
(B^rA^pB^r)^{\frac{1}{q}}\geq B^{\frac{p+2r}{q}} 
\end{equation}
and 
\begin{equation}
A^{{\frac{p+2r}{q}}}\geq (A^rB^pA^r)^{\frac{1}{q}}.
\end{equation}
\end{Theorem}
\noindent Fortunately, Theorem \ref{H-M-L}. (ii) sheds light on the proof of Theorem F in \cite{Mond} and similarly proves the quaternionic Furuta inequalities.
Similarly, Theorems 7.2, 7.3, 7.4, 7.5, and Theorem G in \cite{Mond} hold for positive quaternionic operators on a right quaternionic Hilbert space. 
\begin{Theorem} Let $p\in (0, \frac{1}{2})$ and $T=U\vert T\vert \in \mathfrak{B}(\mathsf{H})$ be a \( p \)-hyponormal operator. Then $\tilde{T}$ is $(p+\frac{1}{2
})$-hyponormal.
\end{Theorem}
\begin{proof}
The main key to the proof is Theorem \ref{QFI}, and the proof of Theorem 2 in \cite{Alu} works for the quaternionic case.
\end{proof} 
\noindent Theorems \ref{Q-ALU} and  \ref{QFI} imply the next result.
\begin{Corollary}
Under the assumptions of Theorem \ref{QFI}, if $\tilde{T}=\tilde{U}\vert \tilde{T}\vert$ is the polar decomposition of $\tilde{T}$ then, the operator $T=\vert \tilde{T}\vert ^{\frac{1}{2}}\tilde{U}\vert \tilde{T}\vert ^{\frac{1}{2}}$ is hyponormal.
\end{Corollary}

\noindent In the rest of the paper \cite{Alu}, Alutghe proves that if $T=U\vert T\vert$ is a \( p \)-hyponormal operator for $p\in (0, \frac{1}{2})$, then the eigenspaces of $U$ reduce $T$ (Theorem 4 in \cite{Alu}).  In the next theorem, we prove this fact for a quaternionic \( p \)-hyponormal operator.
\begin{Theorem}
Let $T\in \mathfrak{B}(\mathsf{H})$ be \( p \)-hyponormal for $0<p<\frac{1}{2}$. If $T=U\vert T\vert$ is the polar decomposition of $T$ and $U$ is unitary, then for each $q\in \sigma_{pS}(U)$, the subspace $\ker \Delta_q(U)$ is a reducing subspace for $T$.
\end{Theorem}
\begin{proof}
Since $U$ is unitary, by Theorem 4.8 in \cite{Ghi1},  $\sigma_S(U)\subset \{q\in \mathbb{H}, \vert q\vert=1\}$.  If $q\in \sigma_{pS}(U)$, then $\vert q\vert=1$ and by Proposition 4.5 in \cite{Ghi1}, $q$ is a right eigenvalue of $U$. By Definition \ref{spectrum}. (i), $\ker  \Delta_q(U)\neq \{0\}$ thus there exists at least a non-zero $ u\in \ker  \Delta_q(U)$, so that $Uu=uq$ and $U^*u=uq$, as $\Delta_q(U)=\Delta_q(U^*)$,  by Proposition 1 (iii) in \cite{Facompact}. Since $T$ is \( p \)-hyponormal  $0\leq (T^*T)^p-(TT^*)^p$, therefore by Corollary \ref{T^*}. (ii),  right linearity of $T$, and the properties of the  inner product on $\mathsf{H}$, we have
\begin{align*}
& 0 \leq \langle (T^*T)^p-(TT^*)^pu,u\rangle = \langle (\vert T\vert ^{2p}- \vert T^*\vert ^{2p}) u, u\rangle \\
&= \langle \vert T\vert ^{2p}u, u\rangle- \langle U \vert T\vert ^{2p} U^* u, u\rangle 
= \langle \vert T\vert ^{2p}u, u\rangle- \langle \vert T\vert ^{2p} U^* u, U^*u\rangle \\
&= \langle \vert T\vert ^{2p}u, u\rangle- \langle \vert T\vert ^{2p} uq, uq\rangle 
= \langle \vert T\vert ^{2p}u, u\rangle- \langle \vert T\vert ^{2p} u, u\rangle \overline{q}q\\
&=0.
\end{align*}
This implies that $(T^*T)^pu-(TT^*)^pu=0$ or $\vert T\vert ^{2p}u=U \vert T\vert ^{2p} uq$, multiplying from right the both sides of the latter equality by $\overline{q}$,  we have $\vert T\vert ^{2p}u\overline{q}=U \vert T\vert ^{2p} u$ that means $\vert T\vert ^{2p} u\in \ker \Delta_{\overline{q}}(U)=\ker \Delta_q(U)$.
\end{proof}
\noindent As mentioned in  Remark 4.6 in \cite{Ghi1}, the subspace $\ker \Delta_q(T)$ has the role of an eigenspace, so the following corollary is concluded.
\begin{Corollary}
Let $T\in \mathfrak{B}(\mathsf{H})$ be \( p \)-hyponormal for $0<p<\frac{1}{2}$. If $T=U\vert T\vert$ is the polar decomposition of $T$ and $U$ is unitary, then the eigenspaces of $U$ reduce $T$.
\end{Corollary}


\section{Generalized Cauchy-Schwarz Inequality}

According to \cite{Choi}, every \( p \)-hyponormal operator on a complex Hilbert space is paranormal. Choi et al. in \cite{Choi} introduce a new family of operators that set between \( p \)-hyponormal and paranormal operators on a complex Hilbert space. In this section, we follow \cite{Choi} and investigate similar conclusions in the quaternionic setting. 
As before, $\mathsf{H}$ will stand for a right quaternionic Hilbert space, and by $\mathfrak{B}(\mathsf{H})$, we mean the space of bounded right linear quaternionic operators on $\mathsf{H}$. The following definition generalizes Definition 1.1 in \cite{Choi} to the quaternionic case.
\begin{Definition}
An operator $T\in \mathfrak{B}(\mathsf{H})$ is said to satisfy the generalized Cauchy-Schwarz inequality, and we write $T\in GCSI(\mathsf{H})$ if $T$ satisfies the following inequality:
\begin{equation}\label{GCSI}
|\langle Tx,y\rangle |\leq (\Vert Tx\Vert \Vert y\Vert)^\alpha  (\Vert Ty\Vert \Vert x\Vert)^\beta,
\end{equation}
for every $x, y\in \mathsf{H}$ and for some $\beta \in (0, 1]$ and $\alpha +\beta =1$. 
\end{Definition}
\noindent In particular, if $T=I$, the inequality (\ref{GCSI}) is reduced to the Cauchy-Schwarz inequality, which was proved in Proposition 2.2 in \cite{Ghi1}.
\begin{Proposition}\label{properties}
Let $T\in GCSI(\mathsf{H})$,  then the following statements hold
\begin{itemize}
\item[(i)] For any $r\in \mathbb{R}, rT=Tr\in GCSI(\mathsf{H})$.
\item[(ii)] If T is invertible, then $T^{-1}\in GCSI(\mathsf{H})$.
\item[(iii)] If $S$ is unitalily equivalent to $T$, then $S\in GCSI(\mathsf{H})$.
\item[(iv)] If $\mathcal{M}$ is a right-invariant subspace for $T$, then $PTP=TP \in GCSI(\mathsf{H})$ where $P$ is the orthogonal projection of $\mathsf{H}$ onto $\mathcal{M}$.
\end{itemize}
\end{Proposition}
\begin{proof}
\begin{itemize}
\item[(i)] Let $q\in \mathbb{H}$. According to Section 3.4 in \cite{Ghi1},  $qT$ is a bounded right linear operator which is defined by $(qT)x:=q(Tx)$ for $x\in \mathsf{H}$.
For a Hilbert basis $\mathsf{N}$ of $\mathsf{H}$ by (\ref{left multiplication}) 
$$q(Tx)= \sum_{z\in \mathsf{N}}zq\langle z,Tx\rangle.$$
Therefore,
\begin{align*}
\vert \langle (qT)x,y\rangle \vert &= \vert \langle q(Tx),y\rangle \vert=\vert \langle \sum_{z\in \mathsf{N}}zq\langle z,Tx\rangle ,y\rangle \vert= \vert\sum_{z\in \mathsf{N}} \langle zq\langle z,Tx\rangle ,y\rangle \vert \\
&= \vert\sum_{z\in \mathsf{N}}  \overline{q\langle z,Tx\rangle }\langle z,y \rangle\vert= \vert \sum_{z\in \mathsf{N}} \overline{\langle z,Tx\rangle } \overline{q}\langle z, y \rangle \vert 
=  \vert \sum_{z\in \mathsf{N}} \langle Tx, z\rangle  \overline{q}\langle z, y \rangle \vert  \\
&= |r| |\langle\sum_{z\in \mathsf{N}}  z \langle z,Tx\rangle ,y\rangle | \tag*{$q=r\in \mathbb{R}$}\\
&= |r| |\langle Tx,y\rangle | \leq  |r|  (\Vert Tx\Vert \Vert y\Vert)^\alpha  (\Vert Ty\Vert \Vert x\Vert)^\beta \\
&= (\Vert r (Tx)\Vert \Vert y\Vert)^\alpha  (\Vert r(Ty)\Vert \Vert x\Vert)^\beta \\
&= (\Vert (rT)x\Vert \Vert y\Vert)^\alpha  (\Vert (rT)y\Vert \Vert x\Vert)^\beta .
\end{align*}
\item[(ii)] Let $T\in GCSI(\mathsf{H})$ be an invertible operator and $x, y\in \mathsf{H}$ be arbitrary. By Remark 2.16 in \cite{Ghi1}, $T^*$ is bijective and  $(T^*)^{-1}=(T^{-1})^*\in \mathfrak{B}(\mathsf{H})$. Also, by Proposition 2.19 in \cite{Ghi1}, $\Vert Tu\Vert =\Vert T^*u\Vert$ for all $u\in\mathsf{H}$. Now, let $y=T^*u$ for some  $u\in\mathsf{H}$
\begin{align*}
|\langle T^{-1}x,y\rangle|&= |\langle x,(T^{-1})^*y\rangle|=|\langle x,(T^{-1})^*T^*u\rangle|=|\langle x,(T^*)^{-1}T^*u\rangle|=|\langle TT^{-1}x,u\rangle|\\
&\leq (\Vert TT^{-1}x\Vert \Vert u\Vert)^{\alpha}(\Vert T^{-1}x\Vert \Vert Tu\Vert)^{\beta}\\
&= (\Vert x\Vert \Vert (T^*)^{-1}y\Vert)^{\alpha}(\Vert T^{-1}x\Vert \Vert T^*u\Vert)^{\beta}
=(\Vert x\Vert \Vert (T^*)^{-1}y\Vert)^{\alpha}(\Vert T^{-1}x\Vert \Vert y\Vert)^{\beta}\\
&=(\Vert x\Vert \Vert (T^{-1})^*y\Vert)^{\alpha}(\Vert T^{-1}x\Vert \Vert y\Vert)^{\beta}= (\Vert x\Vert \Vert T^{-1}y\Vert)^{\alpha}(\Vert T^{-1}x\Vert \Vert y\Vert)^{\beta}.
\end{align*}
\item[(iii)] Let $S$ be unitalily equivalent to $T$ and $U$ be the unitary operator so that $S=U^*TU$. Thus, by Definition 2.25 in \cite{Colombo}, $U^*=U^{-1}$, therefore $\Vert U \Vert=1$. For $x, y\in\mathsf{H}$ 
\begin{align*}
|\langle Sx,y\rangle|&= |\langle U^*TUx,y\rangle|= |\langle TUx,Uy\rangle|\\
&\leq   (\Vert TUx\Vert \Vert Uy\Vert)^{\alpha}(\Vert Ux\Vert \Vert TUy\Vert)^{\beta}\\
&= (\Vert  UU^* TUx\Vert \Vert Uy\Vert)^{\alpha}(\Vert Ux\Vert \Vert UU^* TUy\Vert)^{\beta}\\
&=  (\Vert  U^* TUx\Vert \Vert Uy\Vert)^{\alpha}(\Vert Ux\Vert \Vert U^* TUy\Vert)^{\beta}\\
&=  (\Vert  Sx\Vert \Vert y\Vert)^{\alpha}(\Vert x\Vert \Vert Sy\Vert)^{\beta}.
\end{align*}
\item[(v)] Let $\mathcal{M}$ be a  right-invariant subspace for $T$ which means $\mathcal{M}$ is a closed right-subspace of $\mathsf{H}$ and $T\mathcal{M}\subset \mathcal{M}$, and let $P$ be the orthogonal projection of $\mathsf{H}$ onto $\mathcal{M}$. Therefore, $Px$ is the unique point in $\mathcal{M}$ such that $x-Px\perp \mathcal{M}$ for any $x\in\mathsf{H}$. Since $\mathcal{M}=\{Px, x\in\mathsf{H}\}$, it is clear that $T\mathcal{M} \subset \mathcal{M}$ if and only is $TPx=PTPx$. In \cite{Fa}, Fashandi et al. showed that $P$ is a right linear operator which is an idempotent, $\Vert Px\Vert \leq \Vert x\Vert$ for all $x\in \mathsf{H}$,  $\ker P=\mathcal{M}^{\perp}$, and $ran P= \mathcal{M}$. Now, we have
\begin{eqnarray*}
 |\langle TPx, y\rangle|\leq  (\Vert  TPx\Vert  \Vert  y\Vert)^{\alpha}(\Vert Px\Vert \Vert Ty\Vert)^{\beta}.
\end{eqnarray*}
For all $x\in\mathsf{H}$, $\Vert Px\Vert \leq \Vert x\Vert$, thus $\Vert Px\Vert^{\beta} \leq \Vert x\Vert^\beta$. Also, if $y\in \mathcal{M}$ then for some $y_1\in \mathsf{H}, Ty=TPy_1=TP^2y_1=TPy$, hence
\begin{eqnarray*}
 |\langle TPx, y\rangle |\leq  (\Vert  TPx\Vert  \Vert  y\Vert)^{\alpha}(\Vert x\Vert \Vert TPy\Vert)^{\beta},
\end{eqnarray*}
and if  $y\in \mathcal{M}^\perp$, then $\langle TPx| y\rangle=0$. Thus, we have the following inequality 
\begin{eqnarray*}
 |\langle TPx, y\rangle |\leq  (\Vert  TPx\Vert  \Vert  y\Vert)^{\alpha}(\Vert x\Vert \Vert Ty\Vert)^{\beta}.
\end{eqnarray*}
\end{itemize}
\end{proof}
\begin{Remark}
In Proposition \ref{properties} (i),  this question remains what the result is,  if $q\in \mathbb{H}-\mathbb{R}$.
\end{Remark}
\noindent The next theorem is the quaternionic version of Theorem 2.3 in \cite{Choi}. 
\begin{Theorem}
If $T \in  \mathfrak{B}(\mathsf{H})$ is \( p \)-hyponormal, then $T\in GCSI(\mathsf{H})$.
\end{Theorem}
\begin{proof}
Let $T=U|T|$ be the polar decomposition of $T$. Since $T$ is \( p \)-hyponormal, obviousely, $|T^*|^{2p}\leq |T|^{2p}$ for $0<p\leq 1$. Let $x, y\in\mathsf{H}$ and $q=1-p$, then we have
\begin{align*}
|\langle Tx, y\rangle|^2 &= |\langle U|T|x, y\rangle|^2=|\langle |T|x, U^*y\rangle|^2 = |\langle |T|^{p+q}x, U^*y\rangle|^2 =|\langle |T|^{q}x, |T|^{p}U^*y\rangle|^2\\
&\leq  \Vert |T|^{q}x\Vert^2\Vert |T|^{p}U^*y\Vert^2 = \Vert |T|^{q}x\Vert ^2\Vert U|T|^{p}U^*y\Vert^2 \\
&\leq  \langle |T|^{q}x, |T|^{q}x\rangle  \langle U|T|^{p}U^*y, U|T|^{p}U^*y\rangle \\
&= \langle |T|^{q}x, |T|^{q}x\rangle  \langle |T^*|^{p}y, |T^*|^{p}y\rangle  \tag*{by Corollary \ref{T^*}. (i)}\\
&= \langle |T|^{2q}x, x\rangle  \langle |T^*|^{2p}y, y\rangle\\
&\leq \langle |T|^{2q}x, x\rangle  \langle |T|^{2p}y, y\rangle\\
&\leq  \langle |T|^{2}x, x\rangle ^q \Vert x\Vert ^{2(1-q)}  \langle |T|^{2}y, y\rangle^p \Vert y\Vert ^{2(1-p)} \tag*{by Theorem \ref{H-M-L}. (b)} \\
&= \langle |T|x, |T|x\rangle ^q \Vert x\Vert ^{2(1-q)}  \langle |T|y, |T|y\rangle^p \Vert y\Vert ^{2(1-p)}\\
&= \Vert |T|x\Vert ^{2q}  \Vert x\Vert ^{2(1-q)} \Vert |T|y\Vert ^{2p} \Vert y\Vert ^{2(1-p)}\\
&= \biggl ( \Vert Tx\Vert \Vert y\Vert)^q ( \Vert Ty\Vert \Vert x\Vert)^p \bigg)^2.
\end{align*}
\end{proof}
\noindent Theorem \ref{reducing} is an extended version of Theorem 2.6 in \cite{Choi}  to the quaternionic case. Since the proof is similar to the complex case, we will omit it for brevity.
\begin{Theorem}\label{reducing}
If $T\in GCSI(\mathsf{H})$, then $\ker T\subset \ker T^*$ and $\ker T= \ker T^2$. Therefore, $\ker T$ is a reducing subspace for $T$. 
\end{Theorem}
\noindent In \cite{Choi} it is proved that all bounded operators on a complex Hilbert space satisfying the generalized Cauchy-Schwarz inequality ae paranormal. We aim to prove the same fact for the quaternionic bounded operators. To do this, we need to prove the following theorem: the extension of Theorem 2.12 in \cite{Choi}.
\begin{Theorem}\label{Polar}
Let $T=U\vert T\vert$ be the polar decomposition of $T$ in $\mathfrak{B}(\mathsf{H})$. If $T\in GCSI(\mathsf{H})$, then for every $x\in\mathsf{H}$,
\begin{equation}\label{TU^*x}
\Vert TU^*x\Vert ^2\leq \Vert T^2U^*x\Vert \Vert U^*x\Vert.
\end{equation}
\end{Theorem}
\begin{proof}
Let $T=U\vert T\vert $ be the polar decomposition of $T$, since $\vert T\vert$ is self adjoint, $T^*=\vert T\vert U^*$ and so $\ker T^*=\ker U^*$. Hence if $x\in  \ker T^*$ the inequality (\ref{TU^*x}) is obviously valid. Let $x\in  (\ker T^*)^\perp$,  by formula (2.14) in \cite{Ghi1} that says  $\Vert \vert T\vert u\Vert= \Vert Tu\Vert$, for all $u\in \mathsf{H}$, we infer
\begin{align*}
0 &\leq  \vert \langle  \vert T^* \vert x, \vert T^* \vert x\rangle\vert ^2 =\vert \langle  U\vert T \vert U^* x, \vert T^* \vert x\rangle\vert^2 \tag*{by Corollary \ref{T^*}. (i)} \\
&= \vert \langle  T U^* x, \vert T^* \vert x\rangle\vert^2 \\
&\leq (\Vert T U^* x\Vert \Vert \vert T^* \vert x\Vert )^{2\alpha} (\Vert T\vert T^* \vert x \Vert \Vert  U^* x\Vert)^{2\beta} \tag*{by (\ref{GCSI})}\\
&=  (\langle T U^* x, T U^* x\rangle \Vert \vert T^* \vert x\Vert^2)^{\alpha} (\langle T\vert T^* \vert x, T\vert T^* \vert x \rangle \Vert U^* x\Vert^2)^{\beta}\\
&=  (\langle UT^*T U^* x,  x\rangle \Vert \vert T^* \vert x\Vert^2)^{\alpha} (\langle \vert T^* \vert T^*T\vert T^* \vert x, x \rangle \Vert U^* x\Vert^2)^{\beta}\\
&=  (\langle U \vert T\vert ^2 U^* x,  x\rangle \Vert \vert T^* \vert x\Vert^2)^{\alpha} (\langle \vert T^* \vert \vert \vert T\vert ^2\vert T^* \vert x, x \rangle \Vert U^* x\Vert^2)^{\beta}\\
&=  (\langle \vert T^*\vert ^2x,  x\rangle \Vert \vert T^* \vert x\Vert^2)^{\alpha} (\Vert \vert T\vert  \vert T^*\vert x \Vert \Vert U^* x\Vert)^{2\beta}\\
&=  (\langle \vert T^*\vert x,  \vert T^*\vert x\rangle \Vert \vert T^* \vert x\Vert^2)^{\alpha} (\Vert \vert T\vert  \vert T^*\vert x \Vert \Vert U^* x\Vert)^{2\beta}.
\end{align*}
The latter inequality implies that
\begin{equation*}
\vert \langle  \vert T^* \vert x, \vert T^* \vert x\rangle\vert ^{2-\alpha}\leq \Vert \vert T^* \vert x\Vert ^{2{\alpha}}(\Vert \vert T\vert  \vert T^*\vert x \Vert \Vert U^* x\Vert)^{2\beta}.
\end{equation*}
Thus,
\begin{equation*} 
(\Vert \vert T^* \vert x\Vert ^2)^{2\beta}\leq (\Vert \vert T\vert  \vert T^*\vert x \Vert \Vert U^* x\Vert)^{2\beta}.
\end{equation*}
Again by  Corollary \ref{T^*}. (i) and the equalities  $T=U\vert T\vert $ and $\Vert Uu\Vert= \Vert u\Vert $ we obtain,
\begin{equation*}
\Vert U\vert T\vert  U^* x\Vert^2 \leq \Vert \vert T\vert   U\vert T\vert  U^* x \Vert \Vert U^* x\Vert. 
\end{equation*}
Consequently 
\begin{equation*}
\Vert T U^* x\Vert^2 \leq \Vert U \vert T\vert   T  U^* x \Vert \Vert U^* x\Vert= \Vert  T^2  U^* x \Vert \Vert U^* x\Vert.
\end{equation*}
\end{proof}
\noindent As a result of Theorems \ref{reducing} and \ref{Polar}, we obtain the following corollary. The proof is similar to the proof in complex case, which is referable in Corollary 2.13 in \cite{Choi}.
\begin{Corollary}
If  $T\in GCSI(\mathsf{H})$, then $T$ is paranormal.
\end{Corollary}
 
\bibliographystyle{amsplain}

\end{document}